\documentclass[12pt]{amsart}
 \usepackage{amsmath,amscd,latexsym,verbatim,amssymb}
   \usepackage[english]{babel}     
 \usepackage{times} 

 \newcommand{\resp}{{\it resp.} }
\newcommand{\cf}{{\it cf.} }
\newcommand{\ie}{{\it i.e.} }
\newcommand{\eg}{{\it e.g.} }
\newcommand{\loccit}{{\it loc. cit.} }

\newcommand{\sC}{\mathcal{C}}
\newcommand{\sD}{\mathcal{D}}
\newcommand{\sE}{\mathcal{E}}
\newcommand{\sF}{\mathcal{F}}
\newcommand{\sG}{\mathcal{G}}
\newcommand{\sS}{\mathcal{S}}
\newcommand{\sL}{\mathcal{L}}
\newcommand{\sM}{\mathcal{M}}

\newcommand{\sO}{\mathcal{O}}
\newcommand{\sT}{\mathcal{T}}

\newcommand{\sU}{\mathcal{U}}

\newcommand{\sX}{\mathcal{X}}
\newcommand{\sY}{\mathcal{Y}}
\newcommand{\sZ}{\mathcal{Z}}

\newcommand{\C}{\mathbf{C}}
\newcommand{\Q}{\mathbf{Q}}

\newcommand{\F}{\mathbf{F}}
\newcommand{\N}{\mathbf{N}}
\newcommand{\R}{\mathbf{R}}
 \newcommand{\inj}{\hookrightarrow}

\renewcommand{\epsilon}{\varepsilon}
 at9pt

\newcounter{spec}

\swapnumbers

\newtheorem{thm}{Theorem}[subsection]
\newtheorem{lemma}[thm]{Lemma}

\newtheorem{prop}[thm]{Proposition}

\theoremstyle{definition}
\newtheorem{defn}[thm]{Definition}

\newtheorem{ex}[thm]{Example}

\newtheorem{rem}[thm]{Remark}

\numberwithin{equation}{section}

 at13pt
 at11pt

\begin{document}  

 \title[Uniform sheaves and differential equations]{Uniform sheaves and differential equations}
 \author{Yves
Andr\'e}

\address{D\'epartement de math\'ematiques, Ecole Normale Sup\'erieure\\ \break 45 rue d'Ulm, 75005
Paris\\France.}
\email{yves.andre@ens.fr}

   \bigskip\bigskip\bigskip


 \medskip \begin{abstract}  Real blow-ups and more refined ``zooms" play a key role in the analysis of singularities of complex-analytic differential modules.
They do not change the underlying topology, but the uniform structure. 

This suggests to revisit the cohomology theory of differential modules with help of a suitable new notion of {\it uniform sheaves} based on the uniformity rather than the topology.
 We also investigate the $p$-adic situation (in particular, the analog of real blow-ups) from this uniform viewpoint. \end{abstract} 
\bigskip
  
   \maketitle
    \tableofcontents

  \begin{sloppypar}

\newpage
 \section*{Introduction} 
{\flushright{ \begin{quote}{\small{ ``Les points ne commenceront  \`a peser
que lorsqu'on saura les capter correctement,
non pas comme des figures g\'eom\'etriques,
mais bel et bien comme des puissances
d'explosion. C'est ainsi qu'il faut comprendre
le calcul diff\'erentiel."}} \end{quote}}}

\smallskip\rightline{\small{G. Ch\^atelet  \cite[p. 135]{Cha}.}}

\bigskip\bigskip
 
\subsection{} The study of singularities of a linear differential equation $Ly=0$ with meromorphic coefficients is a rich and delicate topic (\cf \eg \cite[ch. 7, 9]{PS}), which has its roots in the works of Fuchs and Poincar\'e.

One first localizes the problem over an open disk $D$ where the coefficients are analytic, except at the origin where they may have a pole.  
 One may then pass to the formal completion, \ie consider the coefficients of $L$ as formal power series; the study becomes purely algebraic, and it turns out that, after ramification, the differential operator factors as a product of differential operators of order one. 

On the other hand, one may restrict to the punctured disk $D^\ast$, \ie consider the coefficients of $L$ as analytic functions outside the origin; the study becomes purely topological, the differential operator being controlled by the local monodromy.

In order to complete the study of $L$, one then needs a bridge between the (algebraic) formal theory and the (topological) theory over $D^\ast$, which is provided by the theory of asymptotic expansions. This requires some kind of ``zoom" on the singularity ({\it real blow-up} \cite{Sa}\cite{PS}, Deligne's halos \cite{De2}\cite{LP} or subanalytic site \cite{KS}\cite{Mo}) and techniques from {\it sheaf theory}.
  
  This paper grew out of a reflection on the nature of such ``zooms". Restricted to $D^\ast$, the real blow-up does not change the topology (nor the real-analytic structure), but changes the {\it uniform structure} (\cf 2.1). This suggests to revisit the theory of meromorphic differential equations, replacing sheaves by a new notion of {\it uniform sheaves} based on the uniformity rather than the topology.  
 
\subsection{}  Uniform structures (or uniformities) were invented by A. Weil in order to axiomatize, in a qualitative way, the properties of the $\epsilon$-neighborhoods of the diagonal which occur in the definition of uniform continuity, Cauchy sequences and completeness, on a metric space (we recall the main definitions and properties of uniformities, and related proximities, in  \S 1). Soon after Weil, J. Tukey reformulated the notion of uniformity in terms of  ``uniform coverings". In this introduction, we restrict our attention to {\it precompact} uniform spaces $X$, those for which any uniform covering admits a finite subcovering (this amounts to the compacity of the completion $\hat X$). 

According to whether one wishes to patch uniformly compatible local data, or to patch compatible uniform local data, one is led to two different notions of uniform sheaves (\S 3). The first one is more useful and is the one which we develop in detail in this paper: assuming that $X$ is precompact, one defines a Grothendieck topology using the open subsets of the topological space underlying $X$ and the uniform open coverings\footnote{in the non precompact case, the definition has to be modified, \cf  3.1.}; a uniform sheaf on $X$ is a sheaf for this Grothendieck topology. 

For instance, bounded functions and uniformly continuous functions give rise to uniform sheaves (of course, they do not form sheaves in the usual sense).

\smallskip We prove that the topos of uniform sheaves on $X$ is equivalent to the topos of sheaves on $\hat X$. It is functorial with respect to uniformly continuous maps, and depends only on the proximity space underlying $X$.

 \subsection{}  When $X$ is a complex smooth open algebraic curve, endowed with the ``sectorial uniformity", a useful example $\sO^{mod}_X$ of uniform sheaf is provided by analytic functions with moderate growth. 
 
 Revisiting the cohomology theory of meromorphic differential modules from the uniform viewpoint, we show that the De Rham cohomology of an algebraic connection on $X$ can be identified with moderate uniform De Rham cohomology, and also with the cohomology of the constructible uniform sheaf of solutions of connection in $\sO^{mod}_X$ (4.1.1). This leads to a ``combinatorial" interpretation of Deligne's index formula, without ever leaving the topological space $X(\C)$. We briefly mention finer ``zooms".
 
  \subsection{}  We also investigate the $p$-adic situation from the uniform viewpoint. We exhibit an analogy between a real blow-up (viewed as a change of uniformity which, at the level of completions, adds a circle of tangential base-points), and the passage from Berkovich analytic spaces to Huber adic spaces. Our starting point is an observation by J. Rivera-Letelier and M. Baker: given an affinoid space $X= {\rm{Spm}}\,A$, the collection of its rational domains can be used to define uniform coverings (rather than to build a Grothendieck topology, as rigid geometry does);  the completion of $X$ for this uniformity is then nothing but the Berkovich space $X^{an}$ attached to $X$. 
  
  On the other hand, in dimension one, adification adds ``circles" of tangential points. Would then the Huber space $X^{ad}$ attached to $X$ be a completion of $X$ for a finer uniformity?
  
   No, because of the dissymetry between non-closed points and closed tangential points in their closure. Yes, if one accepts to drop the symmetry condition in the definition of a uniformity (2.4.1). This provides a nice instance of the well-developed theory of quasi-uniformities (\cf \cite{Ku}). We analyse the corresponding uniform topoi, and discuss overconvergence from the adic viewpoint.
  
  \smallskip At the end of the paper, we discuss in some detail the cohomology theory of $p$-adic differential modules and the $p$-adic index formula, 
  taking the complex situation as a guiding thread.

 \medskip {\it Acknowledgements.}  This paper owes much to the inspiring conversations which I have enjoyed over the years with Francesco Baldassarri, and it is a pleasure to dedicate it to him with friendship and gratitude. 

 I also thank Lorenzo Ramero for discussions, a decade ago, by which I became aware of the uniform nature of the real blow-up and of the significance of Huber's local monodromy.
 \bigskip 

\bigskip 

\bigskip 
 \section{Uniform structures}
 \subsection{Uniformity and proximity}
 General topology formalizes the notions of closeness (to a given point or subspace) and continuity in a qualitative way, independently of any distance.
 
 In the mid 30s, two independent attempts were made to formalize the notions of relative closeness (of points, or subspaces) and uniform continuity in a qualitative way: Weil's theory of uniform spaces (and its reformulation by Tukey), and Efremovich's theory of proximity spaces\footnote{already foreseen by F. Riesz at the Bologna congress in 1908.}. We refer to \cite{BHH} for a historical survey\footnote{see also \cite{Au} about filters and uniformities according to Cartan and Weil. Commenting \cite{W} in 1979, Weil wrote: ``Avec le recul que donnent les quarante derni\`eres ann\'ees, on sourira sans doute du z\`ele que j'apportais alors \`a l'expulsion du d\'enombrable".}, and to \cite[ch. 2 and 9]{Bou}\cite{I} and \cite{NW} as references for the results recalled below.

\begin{defn}\label{def1}
A \emph{uniformity} (\`a la Weil \cite{W}\cite{Bou}) on a set $X$ is a filter $\sU$ on $X\times X$ which is {\it reflexive} (\ie all $E\in \sU$ contain the diagonal $\Delta_X$), {\it symmetric} (\ie $E\in \sU $ iff $E^{-1}\in \sU$) and {\it cotransitive} (\ie for any $E\in \sU$, there is $E'\in \sU$ such that $E'\circ E'\subset E$). The elements of $\sU$ are called {\it entourages}\footnote{viewing them as a binary relations (= correspondences) between elements of $X$, we use the language of correspondences: we write $\circ$ for their composition, $E^{-1}$ for $\{(y,x)\mid (x,y)\in E\}$, and $E(x)$ for $\{y\in X, (x,y)\in E\}$. More generally, if $A\subset X$, $E(A)$ stands for $\{y\in X, \exists x\in A, (x,y)\in E\}$.}.\end{defn}

\begin{defn}  A \emph{uniformity} (\`a la Tukey \cite{T}) on a set $X$ is a family $\sT$ of coverings of  $X$ by subsets, which is stable under {\it finite intersection}, {\it contains any covering which is refined by a covering in $\sT$}, and such that {\it any $(A_i)_i\in \sT$ has a star-refinement} in $\sT$ (\ie a refinement $(B_j)_j $ such that for any $j_0$, the union of all $B_j$'s  meeting $B_{j_0}$ is contained in some $A_i$).\end{defn} 

\smallskip These definitions are actually equivalent: given $\sU$, $\sT$ is the set of coverings which admit a refinement of the form   $(E(x))_{x\in X}$ for some $E\in \sU$; conversely, given $\sT,\,\sU$ is the filter on $X\times X$ with basis $\cup_i \,(A_i\times A_i)$, for $(A_i)_i\in \sT$.

\begin{defn}  A \emph{proximity} (\cf \cite{NW}) on a set $X$ is a binary relation $\nu$ between non-empty subsets of $X$ (``$A$ is {\it near} $B$"), which is {\it reflexive}, {\it symmetric}, and {\it distributive} (\ie $A\,\nu\, (B\cup C)$ iff $A\,\nu\, B$ or $A\,\nu\, C$)$\,$ [one often also adds the axiom: $(\forall C, (A\,\nu\, C) \,{\rm{or}}\, \,(B\,\nu\,  (X\setminus C))) \Rightarrow  A\,\nu\,  B$]. \end{defn}

A proximity can alternatively be described in terms of $\nu$-neighborhoods ($B$ is a $\nu$-neighborhood of $A$ iff $A$ is not near the complement of $B$).

\medskip One has the following logical relations between these notions:

\smallskip \centerline{Metric $\to$ uniformity $\to$ proximity $\to$ topology.} 

\smallskip Let us describe these relations.

The proximity $\nu$ attached to a uniformity $\sU$ is defined by $\,A\,\nu \,B$ iff $ (A\times B)\cap E \neq \emptyset$ for all $E\in \sU\,$  (equivalently, iff $A$ and $B$ meet some $A_i$ for any $(A_i)_i \, \in \sT$).

If $\sU$ comes from a metric, \ie if the $\epsilon$-neighborhoods (in the metric sense) of the diagonal form a basis of the filter $\sU$, then $A\,\nu \,B$ iff $A$ is at distance $0$ from $B$ (and $B$ is a $\nu$-neighborhood of $A$ iff it contains some $\epsilon$-neighborhood of $A$).

  A proximity defines a topology on $X$: the closure $\bar A$ of a subset $A$ is the set of $x$ such that $\{x\}\,\nu\, A$. If the proximity comes from a uniformity $\sU$, then $\bar A= \cap_{E\in \sU}\, E(A)$. For any $x\in X$, the $E(x)$ for $E\in \sU$  (equivalently, the $\cup_{\{i\mid  x\in A_i\}}\,  A_i $ for $ (A_i)_i\in \sT$ ) form a basis of neighborhoods of $x$.

 \medskip The topology induced by a proximity $\nu$ (\resp a uniformity $\sU$) is {\it Hausdorff} iff $\{x\}\,\nu\,\{y\}\,\Rightarrow x=y$ (\resp $\cap_{E\in \sU}\, E=\Delta_X$).  One can show that a Hausdorff uniformity $\sU$ comes from a metric iff it admits a countable basis.

\begin{defn}  A map $f: X\to Y$ is {\it uniformly continuous} iff for any $E\in \sU_Y,\; (f,f)^{-1}(E)\in \sU_X$ (equivalently, for any $(A_j)_j\in \sT_Y$, $(f^{-1}(A_j)_j\in \sT_X$).\end{defn} 
 It is then {\it $\nu$-continuous} for the induced proximity, in the sense that $A\,\nu B\,\Rightarrow f(A)\, \nu \, f(B)$. Any $\nu$-continuous map is continuous for the induced topology. 
This provides forgetful functors:

\smallskip \centerline{$\{$uniform spaces$\}  \to  \{$proximity spaces$ \}  \to  \{$topological spaces$\}$.} 

\smallskip A topological space is uniformisable iff  any closed set $F$ and any point outside $F$ can be separated by some continuous function. Any Hausdorff proximity space is uniformisable.

A Hausdorff compact topological space $X$ admits a unique uniformity (\resp proximity). Any continuous map $X\to Y$ between a Hausdorff compact space $X$ and a uniform space $Y$ is  uniformly continuous (\cf \cite[ch. 2, \S 4]{Bou}).
  
\subsection{Completion} Let $(X,\sU)$ be a uniform space.  We shall often skip $\,\sU\,$ in the notation. 

\begin{defn} A filter $\sF$ on $X$ is {\it Cauchy} iff  for all $E\in \sU$, there is $x\in X$ such that $E(x)\in \sF\,$ (equivalently, iff 
$ \forall E\in \sU, \, \exists A \in \sF, \,  A\times A\subset E$, \cf \cite[4.2]{MN}).

A uniform space $X$ is {\it complete} iff every Cauchy filter has a cluster point (equivalently, iff it converges). \end{defn} 
 
Any uniformly continuous map from a dense subset of a uniform space $X$ into a complete uniform space $Y$ can be extended (uniquely) into a uniformly continuous map on all of $X$.

Any uniform space $X$ has a unique Hausdorff {\it completion}, namely a universal  Hausdorff complete uniform space $\hat X $ with a uniformly continuous map $\iota: X \to \hat X$. 
 The image $\iota(X)$ is dense in $\hat X$, and the uniformity (\resp topology) of $X$ is the inverse image of that of $\hat X$.

\begin{defn}   A uniform space $X$ is {\it precompact} iff for any $E\in \sU$, $X= E(Z)$ for some finite subset $Z \subset X$  (equivalently, iff every ultrafilter is Cauchy; or else, any $(A_i)_i\in \sT$ has a finite subcovering).\end{defn} 

$X$ is precompact iff $\hat X$ is compact.  The uniformity of $X$ is then uniquely determined by the topology of $\hat X$, and is also determined by the proximity of $X$. 

In fact, on any topological space $X$,  {\it compatible Hausdorff proximities $\nu$ classify compactifications $\bar X$}: given $\bar X$, 
$A\,\nu    \,B \Leftrightarrow \bar A^{\bar X} \cap \bar B^{\bar X}\neq \emptyset$;  conversely, given $\nu$,  the associated (Smirnov)  compactification $\bar X $  is the completion of $X$ with respect to the coarsest uniformity compatible with $\nu$ (\cf \cite[7.7, 12.5]{NW}).

   \bigskip
\section{Examples: real blow-up, Berkovich vs. Huber spaces}  

 {\flushright{ \begin{quote}{\small{``Pas un point [...], mais une petite fl\`eche qui est l\`a et qui jaillit hors du point: c'est ce que j'appelle une fulguration."}} \end{quote}}}
 
\smallskip\rightline{\small{G. Ch\^atelet \cite[p. 141]{Cha}.}}

 \medskip \subsection{Real blow-up and sectorial uniformity}
 Let $\,\Delta^\ast = \{z\in \C, \,0< \vert z\vert\leq 1\}\,$  be the punctured circled unit disk. 

 Viewed as the subset $ \{ (x,y)\in \R^2, \, 0< x^2+y^2\leq 1\} $ of $\R^2$, it inherits a metric uniformity, denoted by  $\Delta^\ast_{(met)} $. The Tukey coverings are those which admit a finite refinement $(A_i)_i$ such that the non-relatively compact $A_i$'s contain a small punctured disc (centered at $0$).
 
\smallskip  On the other hand, viewed in polar coordinates as the subset $\{ ({r}, \theta)\in ]0,1] \times S^1\}$, $\Delta^\ast $ inherits another metric uniformity, the {\it sectorial uniformity} denoted by $\Delta^\ast_{(sec)} $. The Tukey coverings are those which admit a finite refinement $(A_i)_i$ such that the non-relatively compact $A_i$'s contain an open sector (centered at $0$).  
 
Passing from polar to cartesian coordinates  ($x= {r} \cos \theta, y= {r}\sin \theta$) induces a uniformily continuous homeomorphism  $\,\Delta^\ast_{(sec)} \stackrel{\rho}{\to} \Delta^\ast_{(met)} $, as follows from the formula $$(x-x')^2+(y-y')^2= ({r} -{r}')^2 + 2{r}{r}'(1-\cos (\theta -\theta'))\leq ({r} -{r}')^2 +  {r}{r}'(\theta -\theta')^2.$$  Its inverse is not uniformly continuous. The proximities differ: for $\Delta^\ast_{(met)} $, two open sectors (centered at $0$) are always near; for $\Delta^\ast_{(sec)} $, they are near iff they overlap.

\smallskip   The completion of $\Delta^\ast_{(met)} $ is the compact unit disk $\Delta$, while the completion of $\Delta^\ast_{(met)} $ is the compact cylinder $ \Delta^{sec} =  [0,1]\times S^1 $ obtained by adding to $\,\Delta^\ast$ the circle $\{0\}\times S^1\cong \mathbb P^1(\R)$ of {\it tangential base-points} (sectors bissected by  $\theta$  provide a basis of a Cauchy  filter converging to $\theta$ ). One thus has a commutative square which expresses the real blow-up in the category of precompact uniform spaces:
          $$\begin{matrix} \Delta^{sec}  = \widehat {\Delta^\ast_{(sec)}} &\stackrel{\hat\rho}{\to}& \Delta = \widehat {\Delta^\ast_{(met)}}\\  {\uparrow}  &  & \uparrow \\ \Delta^\ast_{(sec)}& \stackrel{\rho}{\to}& \Delta^\ast_{(met)} .  \end{matrix}$$

 \medskip  Real blow-up applies more generally in the situation of the complement $X$ of a divisor with
normal crossings $Z $ in a complex analytic manifold.
In the sequel, we will view real blow-ups as
mere changes of uniform structures on $X$, or even mere changes of proximities, and
 try to avoid manifolds with boundary.

  \subsection{Berkovich vs. Huber spaces}\label{BH}
 
  Let now $\,\Delta^\ast = \{z\in \C_p, \,0< \vert z\vert\leq 1\}\,$  be the $p$-adic punctured circled unit disk. 
It inherits a uniformity induced by the metric of $\C_p$. Its completion is the circled unit disk $\Delta$. 

To cope with the total discontinuity of $\Delta$, {\it rigid geometry} exploits its geometric structure of affinoid space: $\Delta = {\rm Spm} \,\C_p\langle T\rangle$.
 For any {\it affinoid space} $X ={\rm Spm} \, A\,$  (maximal spectrum of a quotient $A$ of the Tate algebra of converging functions on the closed unit polydisk), rigid geometry replaces the naive totally discontinuous topology by a Grothendieck topology $X^{rig}$ defined in terms of {\it rational domains}
$$ X(\frac{f_1,\dots , f_n}{g}) =\{x\in X,\, \vert f_i(x)\vert \leq \vert g(x)\vert\}$$ where $f_i, g\in A$ and $A= \sum Af_i $.

\smallskip  Recently, in the context of $p$-adic dynamics, J. Rivera-Letelier \cite{RL} (for the line) and M. Baker \cite{Bak} (for any affinoid space) introduced the idea of {\it using rational domains\footnote{technically, they use rather open Laurent domains for convenience, but rational domains also work and have better stability properties, \cf \cite[final rem.]{Bak}. On the other hand, the extension of the result to general rigid spaces is not known.} to define a uniform space $X_{(an)}$ rather than a Grothendieck site $X^{rig}$}. 

\begin{prop}\cite{Bak} Coverings of the affinoid space $X$ which admit a finite refinement by rational domains form a Tukey uniformity. The completion of the associated uniform space  $X_{(an)}$  is nothing but the Berkovich analytic space $X^{an}$ attached to $X^{rig}$.  In particular, $X_{(an)}$ is precompact.
 \end{prop}  
 
 \begin{rem} $(\Delta, +)$ is a topological group. The translations are uniformly continuous on $\Delta_{(an)}$, but $+$ itself is not uniformly continuous (it does not extend to $\Delta^{an}\times \Delta^{an}$).  
 \end{rem}
 
\smallskip  In analogy with the complex case mentioned above, and with an eye toward applications to differential equations, one may want to add {\it tangential base-points} to the picture.
  
  A natural framework where $p$-adic tangential base-points exist is Huber's theory of {\it adic spaces}.
 The main differences between the Berkovich space $X^{an}$ and the Huber space $X^{ad}$ attached to an affinoid algebra $A$ is that, while the topological space $X^{an}$  is defined purely in terms of continuous multiplicative semi-norms of rank $1$ (\ie real-valued), $X^{ad}$ is the set of continuous multiplicative semi-norms of any rank, and is endowed with the topology generated by rational domains.   $X^{ad}$ is a spectral space\footnote{\ie quasi-compact and sober, with a basis of quasi-compact open subsets stable under finite intersection (equivalently, homeomorphic to a space ${\rm{Spec }}\, B$, \cf \cite{Ho}).}. The natural inclusion $X^{an}\to X^{ad}$ is not continuous, but admits a continuous retraction $\,X^{ad}\stackrel{\rho}{\to} X^{an}$, by means of which $X^{an}$ is the maximal Hausdorff quotient of $X^{ad}$.
 
 We refer to \cite{Berk1}\cite{Hu1} for $p$-adic analytic and adic spaces respectively. Let us describe briefly the case $\Delta = {\rm Spm} \,\C_p\langle T\rangle$. 
 
 In the ``tree-like" analytic unit disk $\Delta^{an}$, one distinguishes four types of points: 
 
 - type 1: these are the points in $\Delta$ (``ends of the tree").
 
 - type 2: they correspond to the sup-norm on disks of radius $r\in p^{\Q_{\leq 0}}$ (``branching points of the tree"). The root $\eta$ of the tree corresponds to the case $r=1$.
 
 - type 3: they correspond to the sup-norm on disks of radius $r\notin p^{\Q_{\leq 0}}$.
 
 - type 4: they correspond to  nested sequences of disks with empty intersection (``dead-ends of the tree"). 
 
 \smallskip The overconvergent adic unit disk $\Delta^{ad \dagger}$ is obtained by ``attaching" to each point $x$ of type 2  a ``circle" ($\cong \mathbb P^1(\bar\F_p)$) of tangential base-points $\theta$, which are closed and belong to the closure of $x$. The affinoid adic unit disk $\Delta^{ad}$ is the complement in $\Delta^{ad, \dagger}$ of one specific tangential base-point $\vec{\eta\infty}$, which lies in the closure of $\eta$ (in fact $\Delta^{ad, \dagger}$ is the closure of $\Delta^{ad}$ in the adic affine line;  the passage from $\Delta^{ad}$ to $\Delta^{ad \dagger}$ is a special case of a universal construction in the world of adic spaces, \cf \cite[1.10]{Hu1}\cite[5.9, 5.10]{Hu2}).

 \medskip Now the question arises whether like $\Delta^{an}$,  $\Delta^{ad}$ is the completion of $\Delta$ with respect to some uniformity.  The answer is negative. In fact, there is no compatible proximity at all on $\Delta^{ad}$, because of the ``asymmetry" of its topology: if $x$ is a point of type $2$ and $ \theta$ a tangential base-point in the closure of $x$, then $\theta$ would be near $x$ since it is in $\overline{\{x\}}$, but $x$ would not be near $\theta$ since $\{\theta\}$ is closed. The only hope to put a compatible ``uniformity" on adic spaces is thus to relax the notion of uniformity by dropping the symmetry condition.

 \subsection{A detour through asymmetric topology}
  In fact, in the mid 1950s, Nachbin and others have tried to formalize the common properties of uniform spaces and partial orders by {\it dropping the symmetry condition}. This led to the theory of {\it quasi-uniformity} and {\it quasi-proximity} (same definitions as above\footnote{beware that the equivalent statements mentioned in parenteses are no longer equivalent in this generalized setting.}, but without the symmetry condition on $\sU$ and on $\nu$, \cf \cite{MN}). We refer to \cite{Ku} for a historical survey, and for references for the results recalled below. 

\medskip One has forgetful functors:

\smallskip \centerline{$\{$quasi-uniform spaces$\}  \to  \{$quasi-proximity spaces$ \}  \to  \{$topological spaces$\}$}

\smallskip\noindent but now, any topological space becomes quasi-uniformisable - in fact in several functorial ways, for instance via the finest compatible quasi-uniformity. The finest compatible quasi-proximity is given by $A\,\nu\, B \Leftrightarrow A\cap \bar B \neq \emptyset\,$\footnote{it is induced for instance by the (functorial) Pervin quasi-uniformity generated by finite intersections of $(X\times U)\cup ((X \setminus U)\times X)$, for open subsets $U\subset X$, \cf \cite[\S 19.7, 19.14]{NW}).}. 

Any locally quasi-compact space $X$ has a coarsest compatible quasi-uniformity (the K\"unzi quasi-uniformity), generated by finite intersections of $(X\times U)\cup ((X \setminus K)\times X)$, for  $K\subset  U$, $K$ quasi-compact,  $U$ open (\cf \cite[\S 3]{Ku}). 
 If $X$ is spectral (or more generally if $X$ has a basis of quasi-compact open neighborhoods), then the K\"unzi quasi-uniformity is also generated by finite intersections of $(X\times U)\cup ((X \setminus U)\times X)$, for any basis of {\it quasi-compact} open subsets $U\subset X$ (indeed, this quasi-uniformity is clearly coarser than the K\"unzi one, and it induces the same topology), and it is functorial with respect to quasi-compact (= spectral) maps.

\smallskip Any continuous map $X\to Y$ between a quasi-compact quasi-uniform space $X$ and a uniform space $Y$ is (quasi-)uniformly continuous \cite{La}\footnote{the author learned this result just after noticing that the counter-example on p. 55 of \cite{MN} is wrong - a fortunate mistake: proving that (along Lambrinos' theorem but contrarily to what is stated in {\it loc. cit.}) the Pervin quasi-uniformity on $[0,1]$ is finer than the unique compatible uniformity is an instructive exercise for a newcomer in asymmetric topology, which can be done with a sheet of graph paper!}.

Cauchy filters, complete and precompact spaces, are defined as above. A topological space is quasi-compact iff it is complete with respect to every compatible quasi-uniformity \cite[\S 4.15]{MN}. We shall sometimes write $Y= \hat X$ to indicate that $X$ is dense in $Y$ and $Y$ is complete with respect to a quasi-uniformity compatible with that of $X$, even though in this asymmetric setting, such a quasi-uniform space $Y$ is by no means unique.

 \subsection{A $p$-adic analog of real blow-ups} 
 Let $X= {\rm{Spec}}\, A$ be again an affinoid space as above. From the above remarks, one gets:
 
 \begin{prop} Finite intersections of subsets of the form $(X\times U)\cup ((X \setminus U)\times X)$, where $U$ runs through the rational domains in $X^{ad}$, generate the coarsest quasi-uniformity compatible with the topology of $X^{ad}$. It is complete, and the continuous map $X^{ad}\to  X^{an}$ is  (quasi-)uniformly continuous. 
 In particular, the quasi-uniformity  $X_{(ad)}$ on $X$  induced by $X^{ad}$ is precompact.  \end{prop} 

 One thus has a commutative square in the category of precompact quasi-uniform spaces:
          \begin{equation}\label{CD}\begin{matrix}  X^{ad} = \widehat{X_{(ad)}} &\stackrel{\hat\rho}{\to}& X^{an} = \widehat {X_{(an)}}\\  {\uparrow}  &  & \uparrow \\ X_{(ad)}& \stackrel{\rho}{\to}& X_{(an)} .  \end{matrix}\end{equation}
      This  applies in particular to $\Delta $, and the situtation is somehow analogous to the real blow-up setting,  the missing point $0$ being replaced by the missing open unit disk $\Delta^-_\infty$ at $\infty$ ($\Delta = \mathbb P^1(\C_p)\setminus  \Delta^-_\infty$).
 
   \begin{rem} There is a strong analogy between the passage from $X^{rig}$ (with its Grothendieck topology defined in terms of rational domains) to the spectral space $X^{ad}$ on one hand, and the passage from a real-algebraic variety with its semi-algebraic Grothendieck topology to the real spectrum (which is a spectral space) on the other hand. In both cases, this passage does not alter the topoi \cite[1.1.11]{Hu1}\cite[ch. 7]{BCC}. Similarly, there is a spectrum attached to the subanalytic Grothendieck topology of complex analytic manifolds (\cf \cite[\S1]{Mo}). It might be interesting to understand these spectral spaces in terms of quasi-uniformities, especially in view of the role of the subanalytic site in the study of singularities of meromorphic differential modules \cite{KS}\cite{Mo}.
 \end{rem}
  
 \subsection{Formal uniformities and $p$-adic blow-up}\label{form}
On the $p$-adic disk $\Delta$ with its standard formal structure, one can also consider the non-Hausdorff uniformity defined by the equivalence relation $\vert z-z'\vert <1$. Its Hausdorff completion is the discrete line $\mathbb A^1(\bar\F_p)$.   
Blowing up a point in closed fiber $\mathbb A^1(\bar\F_p)$ changes not only the uniform structure (the Hausdorff completion of $\Delta$  becomes   $\mathbb A^1(\bar\F_p)\coprod \mathbb A^1(\bar\F_p)$ ), but also the topology. From this viewpoint, a $p$-adic blow-up is far from being an analog of a real blow-up.
 
More generally, if $Y$ is the Raynaud-Berthelot generic
 fiber of a $\sO_{\C_p}$-formal scheme $\sY$, there is a canonical non-Hausdorff uniform
structure (defined by the tube of the diagonal) for which the completion map is the specialization map $sp: Y \to \vert \sY\vert$, $\vert \sY\vert$ being equipped with the discrete topology.

\bigskip
  \section{Uniform sheaves}
  
  \subsection{Uniform coverings}
  
 Sheaves can be thought as a way of patching compatible local data. Likewise, we shall think of uniform sheaves as a way of {\it patching uniformly compatible local data}, uniformity being understood as a suitable condition on coverings.   
 
 The obvious guess is to use Tukey coverings, but they have a drawback: they don't define a Grothendieck topology in general. For instance,   $(U_i = ]-i, i[)_{i\in \N}$ is a Tukey covering of the (metric) uniform real line $\R$, and for each $i$, $(U_{ij}= ]-i+\frac{j-1}{2}, -i+\frac{j+1}{2}[)_{j=1,\dots, 2i-1}$ is a Tukey covering of $U_i$, but $(U_{ij})_{ij}$ is not a Tukey covering of $\R$.
  For this reason, we shall consider a more flexible notion of {\it uniform covering}. 

 \smallskip Let $X$ be a uniform space, and let $X\stackrel{\iota}{\to} \hat X$ be the Hausdorff completion. 
 
 \begin{defn}
 We say that an open covering $(U_i)_i$ of $X$  is \emph{uniform} iff it is the inverse image by $\iota$ of an open covering of $\hat X$.\end{defn}
 
   This is stronger than asking that $(\hat U_i)_i$ is a covering of $\hat U$: for instance, open sectors $U_i$ (centered at $0$ and of angle $<2\pi$) which cover $\Delta^\ast$ do not form a uniform covering of $\Delta^\ast_{(met)}$ (but they do for $\Delta^\ast_{(sec)}$), alhough their completions form a covering of $\Delta$.  
 
   \smallskip     Tukey open coverings are uniform (indeed, they are the inverse image by $\iota$ of Tukey open coverings of $\hat X$) . The converse is true if $X$ is precompact, in which case any uniform covering admits a finite subcovering.

\begin{lemma}\label{l7}  \begin{enumerate} \item For any open subset $U\subset X$, there is a (unique) biggest open subset $\check U\subset \hat X $  such that $\iota^{-1}(\check U)= U$. 
\item $\check U\subset \hat U = \overline{\iota(U)}$.  If $\iota^{-1}(\hat U)= U$, then $\check U$ is the interior of $\hat U$.
\item $\check U \cap \check U'=  (U\cap U')\check{}\,$.
  \item  For any collection of open subsets $U_i$ of $X$, $\bigcup \check U_i \subset (\bigcup U_i)\check{}\,$.
\item $(U_i)_i$ is a {\emph{uniform}} open covering of a given open subset $U\subset X$ iff $(\check U_i)_i $ is a covering of $\hat U$. 
\item  The topology of any open subset $V\subset \hat X$ has a basis of the form $(\check U_j)$, where $U_j$ is some basis of the topology of $\iota^{-1}(V)$.
\item  Any open subset $V\subset \hat X$ admits a covering of the form $(\check U_i)_i$, where $(U_i)_i$ is some {\emph{uniform}} open covering of $\iota^{-1}(V)$.  
\end{enumerate}
 \end{lemma}

 \begin{proof} (1) $\check U$ is clearly the union of all open subsets $V\subset \hat X$ such that $\iota^{-1}(V)=U$.
 
  (2) For the identification $\hat U = \overline{\iota(U)}\subset \hat X$, we refer to \cite[II.26.9.cor. 1]{Bou}. Replacing $X$ by $\iota(X)$, we have to show that $\check U\subset \bar U$, \ie any non-empty open subset $V\subset \check U$ meets $U$. But $V\cap X$ is non-empty by density of $X$, and is contained in $U$ since $\check U\cap X= U$.
  It follows that $\check U$ is contained in the interior of $\hat U$.   
  
   If $ \hat U\cap X = U$, the reverse inclusion is immediate.
 
 (3) The inclusion $\check U \cap \check U'\subset   (U\cap U')\check{}\,$ follows from    $\iota^{-1}(\check U \cap \check U')= U\cap U'$. The reverse inclusion is immediate.
 
  (4) follows from $\iota^{-1}(\bigcup \check U_i)= \bigcup U_i$.
 
 (5) follows from the fact that if $(V_i)_i$ is a covering of $\hat U$, so is $((\iota^{-1}(V_i))\check{})_i$, taking into account item 2.
 
 (6) We may replace $X$ by $\iota(X)$. We know that $\hat X$ is regular (as any uniform space), and so is its subset $V$. In particular, any point $x\in V$ has a basis of neighborhoods in $V$ which are interiors $W$ of closed subsets $\hat W$ of $\hat X$. By item 2, $W= (W\cap X)\check{}$.
 
 (7) is a straightforward consequence of items (5) and (6).
   \end{proof}

 \begin{rem}\label{rem}  These definitions and arguments extend to the quasi-uniform situation, when $\hat X$ is {\it a} completion of a quasi-uniform space $X$, except items 6 and 7 which uses the regularity of uniform spaces. Actually, only the fact that any point has a basis of open neighborhoods which are interiors of their closures is used. This holds for instance in the situation of an affinoid adic space $X^{ad}$ endowed with its canonical quasi-uniform structure as above. Indeed, any point has a basis of rational domains $W$. Let $\rho: X^{ad}\to X^{an}$ be the canonical map. Then, $\rho(W)$ is  compact, and $\rho^{-1}\rho(W)$ is the closure of $W$ in  $X^{ad}$ and $W$ is the interior of $\rho^{-1}\rho(W)$.  
 
This does not hold, in contrast, for $\Delta^{ad \dagger}$: the maximal point $\eta$ of the disk does not have a basis of open neighborhoods of the form $\check{U},\, U\subset  \Delta$, since $\check U$ always contains $\vec{\eta\infty}$.\end{rem}

   \subsection{Uniform $G$-topology and uniform sheaves}
        
        Let $X$ be again a uniform space.
        
          \begin{defn}
 The \emph{uniform $G$-topology} of $X$ has for open sets the open subsets $U$ of $X$, and the $G$-coverings of $U$ are the uniform coverings of $U$ (viewed as a uniform subspace of $X$). \end{defn}

\begin{lemma} This defines a saturated Grothendieck topology (\cf \cite[9.1.2]{BGR}), which is functorial with respect to uniformly continuous maps.
 \end{lemma} 
 
 Recall that it is possible to glue saturated Grothendieck topologies (\cf \cite[9.1.3]{BGR}).
 
 \begin{proof} This is a {Grothendieck topology} since:

- $(U)$ is a uniform covering of $U$,

- if $V\subset U$ are open subsets, and $(U_i)_i$ a uniform covering of $U$, then  $(U_i\cap V)_i$ a uniform covering of $V$ (indeed $\cup (\check U_i\cap \hat V)= (\cup \check U_i)\cap \hat V= \hat V$),

- if $(U_i)_i$ is a uniform covering of $U$, and $(U_{ij})_j$ is a uniform covering of $U_i$, then $(U_{ij})_{ij}$ is a uniform covering of $U$ (indeed $\cup  \check U_{ij} = \cup_i   \hat U_i = \hat U$).

It is saturated since:

- If $V$ is a subset of the open set $U$ and if there is a uniform covering $(U_i)_i$ of $U$ such that $V\cap U_i$ is open, then $V$ is open,

- any open covering $(U_i)_i$ of $U$ which admits a uniform refinement $(V_j)_j$ is uniform (indeed,  $(\check V_j)_j$ is a refinement of $(\check U_i)_i$, hence both cover $\hat U$).

\smallskip  It is clear that this Grothendieck topology is 
 {functorial 
  with respect to uniformly continuous maps}, since such maps extend to the completions.
 \end{proof}

        \begin{defn}  A \emph{uniform sheaf} is a sheaf for the uniform $G$-topology. We denote by $\tilde X$ the \emph{uniform topos}, \ie the topos of uniform sheaves on $X$.  
        \end{defn}
        
     It is functorial with respect to uniformly continuous maps $X\to Y$.

     \begin{prop}\label{equiv}      For any uniform space $X$, $\iota$ induces a morphism of sites from $X$ (with its uniform $G$-topology) to $\hat  X$, whose associated map of topoi $   \tilde X \stackrel{\tilde\iota}{\cong} \tilde{\hat X}\,$ is an equivalence. A quasi-inverse  $\,\tilde{\hat X} \stackrel{\tilde h}{\cong} \tilde X\,$ is given by 
        $$h_\ast(\sG)(U) =  \sG(\check U),\; h^\ast(\sF)(V)=  \sF(\iota^{-1}(V)).$$  \end{prop}
        
        \begin{proof}  Let $h^{-1}$ be the functor from the site $X$ (with the uniform $G$-topology) to the site $\hat X$ given by $h^{-1}(U)= \check U$. By definition of the uniform $G$-topology, it is continuous, \ie induces by composition a functor $\,\tilde{\hat X} \to   \tilde X\,$ which is $h_\ast$. On the other hand, for any $\sF\in \tilde X, \sG\in \tilde{\hat X}$, one has a canonical map 
       ${\rm{Mor}}\,(h^\ast(\sF) , \sG)\to {\rm{Mor}}\,( \sF , h_\ast(\sG))$ which is clearly injective. It is also surjective by item 7 of the Lemma \ref{l7}: the value of a morphism $h^\ast(\sF) \to \sG$ at $V$ can be written as $\sF(\iota^{-1}(V))= \sF( {(U_j)_j)} \to \sG(V)= \sG({(\check U_j)_j)}$ for a suitable uniform covering $(U_j)_j$ of $V\cap X$, and  can thus be expressed in terms of the values of a morphism $\sF \to h_\ast(\sG)$ 
        at the $U_j$'s. Hence $h^\ast$ is left adjoint to $h_\ast$.  Since $h^\ast(\sF) = \iota_\ast $, it is also a right adjoint, hence is exact. Therefore $\tilde h$ comes from a morphism of sites $h$, which is left quasi-inverse to $\iota$. 
      
      It remains to show that $\tilde h$ is an equivalence. We conclude by using the criterium of 
      SGA 4 III. 4.1: 
    \begin{enumerate} \item $h^{-1}$ respects finite limits (indeed these are just finite intersections of open subsets, and item 3 of Lemma \ref{l7} applies),
  \item for any open $V \subset \hat X$, there is a covering $( V_i \to V)_i$ with $V_i\in {\rm{Im}}\,h^{-1}$ (this is item 7 of the lemma),
  \item if $(h^{-1}(U_i)\to h^{-1}(U ))_i$ is a covering, then so is $(U_i\to U)_i$ (this follows from item 5 of the lemma),
  \item $h^{-1}$ is fully faithful (this is immediate, since morphisms between open subsets are just inclusions).  
  \end{enumerate}\end{proof}
      
  If $X$ is Hausdorff complete, $\tilde X$ is nothing but the topos of sheaves on the underlying topological space $X_{top}$.   In general, there is a canonical morphism from $  \widetilde{X_{top}}$ to $\tilde X$, given by  the composition $  \widetilde{X_{top}} \to  \widetilde{\hat{X}_{top}}  \cong \,\tilde{\hat X} \stackrel{\tilde h}{\cong} \tilde X $.
     
    \smallskip  {\it If $X$ is precompact, the topos $\tilde X$ can be defined without reference to $\hat X$, since the uniform coverings of $U\subset X$ are the Tukey open coverings of $U$}. In fact, in this case, $\tilde X$ depend only on the induced proximity (\cf end of \S 1), and uniform sheaves could also be named {\it proximal sheaves}. 
        
        For instance, {\it bounded} continuous numerical functions (\resp {\it uniformly continuous} functions) form a {\it uniform sheaf} on any precompact uniform space (of course, they do not define a sheaf in the usual sense in general).

   \subsection{Examples}
 
\subsubsection{Real blow-up and sectorial uniformity again}\label{exmod}     Real blow-ups do not change the topology of (open) complex manifolds, but change their uniform $G$-topology: there are more uniform coverings. 
 
 For instance, locally constant sheaves in the usual sense define locally constant sheaves for the
sectorial uniform $G$-topology (but not for the standard metric
uniform $G$-topology). More generally, constructible sheaves are constructible for the sectorial uniform $G$-topology.

\smallskip Let us come back to the case of the punctured disk $\Delta_{(sec)}^\ast$ with its sectorial uniformity. 
By the previous proposition (and the remark following it), uniform sheaves on $\Delta_{(sec)}^\ast$ are ``equivalent" to sheaves on the compact cylinder $\Delta^{sec}$, but can be characterized without reference ot $\Delta^{sec}$.
 However, points of $\Delta^\ast$ are not enough to provide a conservative system of fiber functors on $\widetilde{\Delta_{(sec)}^\ast}$: one has to consider also the tangential points, as the following example shows.

A useful uniform sheaf on $\Delta_{(sec)}^\ast$ is ${\sO^{mod}_{\Delta_{(sec)}^\ast}}$, whose sections on $U$ are given by those holomorphic fonctions on $U$ which have {\it moderate growth} any germ of sector (centered at $0$) contained in $U$. For instance, if $U\subset \Delta^\ast $ is defined by ${\rm Re}\, z>0$, then $e^{-1/z}\in {\sO^{mod}_{\Delta_{(sec)}^\ast}}(U)$. 

The inclusion ${\sO^{mod}_{\Delta_{(sec)}^\ast}}\inj  {\sO_{\Delta^\ast}}$ is a monomorphism of uniform sheaves but not an isomorphism. It induces an isomorphism on fibers at any point of $\Delta^\ast$, but not at tangential points.

\subsubsection{Berkovich and Huber spaces again}  
  Let us first recall that for any complete non-archimedean field $k$, there are full embeddings of categories 
  $$ An_k  \inj  Rig_k \inj  Ad_k,\;\, X^{an}\mapsto X^{rig}\mapsto X^{ad}$$
  between the categories of strictly analytic Hausdorff Berkovich $k$-analytic spaces, rigid-analytic varieties and (analytic) adic spaces over $k$, which respect the subcategories of affinoid objects. 
  
  For any $X^{rig}\in Rig_k $, this induces an equivalence of topoi $  \widetilde{X^{ad}} \cong   \widetilde{X^{rig}} $, and for any $X^{an}\in An_k$, and a morphism $ \widetilde{X^{rig}} \stackrel{\tilde{\hat\rho}}{\to}  \widetilde{X^{an}}$ which, combined with the previous equivalence, induces an equivalence between $ \widetilde{X^{an}}$ and the subtopos of $ \widetilde{X^{ad}}$ consisting of sheaves $\sF$ such that for any pair of points $(x,\theta)$ with $\theta\in \overline{\{x\}}$, $\sF_\theta\cong \sF_x$. Similarly, for the respective \'etale topoi, \cf \cite[1.1.11, 8.3, 2.3]{Hu1}.
  
  \smallskip Let us come back to the case of an affinoid space $X = {\rm{Spm}}\, A$ (over $k=\C_p$, to fix ideas).
  Combining the previous equivalences of topoi with Proposition \ref{equiv} (taking into account Remark \ref{rem}), one obtains:  
  
  \begin{prop}  The commutative square \eqref{CD} induces a commutative square of topoi
           \begin{equation}\label{CD2}\begin{matrix}  \widetilde{X^{rig}}   &  \stackrel{\tilde{\hat\rho}}{\to} & \widetilde{X^{an}}  \\  {\uparrow}\cong  &  & \cong \uparrow  \\ \widetilde{X_{(ad)}}& \stackrel{\tilde\rho}{\to}& \widetilde{X_{(an)}}.  \end{matrix}\end{equation}
   \end{prop} 
   
   For instance, the abelian sheaves $\sO$ and $\Omega^1$ exist in each of these topoi, and correspond to each other. The morphism $d: \sO \to \Omega^1$ is not an isomorphism (failure of the Poincar\'e lemma): it induces an isomorphism on fibers at any point of $X$, but not at points of $X^{an}= \widehat{X_{(an)}}$ which do not lie on $X$.

       \subsection{Another notion of uniform sheaves}\label{ano} 
   Instead of patching uniformly compatible local data, one may want to {\it patch compatible uniform local data}. This leads to a completely different notion of ``uniform sheaves", as objects of the category 
      $2$-$\rm{colim}\,_{E\in \sU}\{$Sheaves on $E\}$.
 
Concretely, objects are pairs $\sF = (E\in \sU, \, \sF_E \in \tilde E)$,
morphisms are defined by   $\rm{Mor}(\sF_1 , \sF_2 ) = \rm{colim}_{E\subset E_1\cap E_2}\,   {\rm Mor}(\sF_{E_1\mid E}, \sF_{E_2\mid E}).$
Global sections are given by $\Gamma \sF  = \rm{colim}_{E\in \sU} \Gamma( \sF_E)$  (which is not the same as  $\Gamma(\delta^\ast \sF)$ since entourages do not
form a basis of neighborhoods of the diagonal
in general). One can define a cohomology theory for such objects in a straightforward way.  
       
       This notion seems less useful (in the area of differential equations which we have in mind) than the other notion of uniform sheaves. On the other hand, it extends in a straightforward way to quasi-uniformities and to generalized uniformities in the sense of Appendix 1.

 \bigskip
  \section{De Rham cohomology and uniform sheaves}

  {\flushright{ \begin{quote}{\small{``En g\'eom\'etrie
alg\'ebrique, un point est beaucoup plus petit
qu'en g\'eom\'etrie analytique: une fonction sur
le compl\'ement a une singularit\'e $1/z^n$ (resp. essentielle).
On peut interpoler entre les deux."}} \end{quote}}}
 
\smallskip\rightline{\small{P. Deligne \cite{De2}.}}

\medskip
 \subsection{Complex case}   
\subsubsection{Algebraic vs. analytic De Rham cohomology}
Let now $X$ be the complement of a finite set $Z$ of points in a projective smooth curve $\bar X$ over $\C$. 
 Let ${\sM}$ be an algebraic coherent module on $X$ with connection $\nabla$. Its algebraic De Rham cohomology is
$$H^\ast_{DR}( \nabla) = {\mathbb H}^\ast ( {\sM} \stackrel{\nabla}{\to} {\sM}\otimes \Omega^1_X).$$ 
 Let $({\sM}^{an}, \nabla^{an})$ be the associated analytic coherent module with connection of the analytic curve $X^{an} $ attached to $X$. By the Poincar\'e lemma, $\nabla^{an}$ is an epimorphism, and we denote its kernel by $\sL^{an}  $ (local system of analytic solutions). 
   The analytic De Rham cohomology is  \begin{equation} H^\ast_{DR}( \nabla^{an}) = {\mathbb H}^\ast ( {\sM}^{an} \stackrel{\nabla^{an}}{\to} {\sM}^{an}\otimes \Omega^1_{X^{an}}) \cong H^\ast(X^{an},  \sL^{an}).\end{equation} 
 
The canonical morphism
 $H^\ast_{DR}( \nabla)  \stackrel{\phi}{\to} 
 H^\ast_{DR}( \nabla^{an}) $  
is an isomorphism if $\nabla$ regular \cite{De1}.
 In the presence of irregular singularities, it is not: one has the Deligne index formula \cite{De1} which involves the irregularities (= Fuchs-Malgrange numbers) at the points of $Z$: 
\begin{equation}\label{chi}\chi_{DR}(\nabla)= \chi_{DR}(\nabla^{an}) + \sum_{x\in Z}\, ir_x (\nabla).\end{equation}
with $\chi_{DR}(\nabla^{an}) = {\rm rk}\,{\sM}\cdot \chi(X^{an})=  {\rm rk}\,{\sM}\cdot \chi_{DR}((\sO_X, d))$.

\subsubsection{The uniform viewpoint}\label{comb}
 
In order to analyse the irregular singularities,  one usually introduces the real blow up of $\bar X^{an}$ (a manifold with boundary) and some sheaves on it, \cf \eg \cite[3.4, 5b]{Sa}. 

Alternatively, one can use the sectorial uniformity $X_{(sec)}$ on $X(\C)$ (obtained by glueing sectorial uniformities around the points in $Z$) and some uniform sheaves. Recall that the uniform $G$-topology of $X_{(sec)}$ has the same open sets as $X(\C)$ but less coverings.

Let  ${\sO^{mod}_{X_{(sec)}}}\subset \sO_{X^{an}}$ be the uniform (or proximal) sheaf of analytic functions with moderate growth around $Z$, \cf \ref{exmod}.  This makes $X_{(sec)}$ a ringed site. 
  One has  $\Gamma ({\sO^{mod}_{X_{(sec)}}})= \Gamma ( \mathcal O_{X})$ (rational functions with poles at $Z $).

 \smallskip Let $(\sM^{mod}, \nabla^{mod})$ be the coherent ${\sO^{mod}_{X_{(sec)}}}$-module with connection obtained from $(\sM, \nabla)$ by 
 tensoring with $\sO^{mod}_{X_{(sec)}}$.  By the moderate Poincar\'e lemma (\cf \cite[3.4, 5b]{Sa} and \cite[7.3]{PS}),  $\nabla^{mod}$ is an  
 epimorphism of uniform sheaves (via the comparison of topoi \ref{equiv}), and we denote its kernel by $\sL^{mod}$ (uniform sheaf of moderate solutions).

 \begin{thm} One has canonical isomorphisms
 \begin{equation} H^\ast_{DR}( \nabla) \cong H^\ast_{DR}( \nabla^{mod})\cong  H^\ast(X_{(sec)},  \sL^{mod}).\end{equation} \end{thm}
  
Moreover $\sL^{mod}$ is a {\it constructible} uniform sheaf on $X_{(sec)}$  (the exponentials $e^{Q(1/{z)}}$ which occur typically in the Turrittin formal decomposition around the points of $Z$ are sections of  ${\sO^{mod}_{X_{(sec)}}}$ in suitable sectors), and this leads to a {\it combinatorial proof of Deligne's index formula}.

 All this is nothing but a translation of \cite[3.4]{Sa} in the language of uniform sheaves (which allows not to leave the topological space $X(\C)$), using the comparison of topoi \ref{equiv}.

 \subsubsection{Finer zooms}
    Complex analytic singularities are not ``isolated", but surrounded by a ``halo" of infinitely near singularities, which Deligne's disc unfolds and materializes \cite{De2}.
   This disc is obtained by compactifying $\mathbb C$ with a circle: $ \mathbb C \cup (S^1\times ]0,\infty])/ ( \mathbb C^\ast \sim (S^1\times ]0,\infty[), $ 
 and is endowed with a sheaf of rings $\tilde{\mathcal O}$, following the rule 
 
$$I\times ]k', k''[ \mapsto 
  \frac{\{{\rm{analytic\,functions\, with\, exp.\, growth}}  \leq k' \,{\rm{in}}\,  I\}}
           {\{{\rm{analytic\, functions\, with\, exp.\, decay}} \geq  k''  \,{\rm{in}}\,I\}}.$$
Via asymptotic expansions, global sections of $\tilde{\mathcal O}$ in $k'<k$ can be identified with Gevrey series of order $1/k$.
The definition sets of exponentials $e^{Q(1/z)}$ as sections of $\tilde{\mathcal O}$ give rise to ``Deligne's daisies". Counting their petals leads to a combinatorial proof of refined Gevrey index formulas \cite{LP}.

\smallskip On the other hand, the subanalytic site also allows a fine analysis of singularities: for instance, it allows to recover, using the sheaf of tempered solutions, the exponentials $e^{Q(1/z)}$ which occur in the formal decomposition (up to a multiplicative scalar) \cite{Mo}.

\smallskip It would be interesting to describe these finer zooms in terms of suitable uniform sheaves for a suitable uniform structure on $X$.

 \subsection{$p$-adic case} 
 
  \subsubsection{Algebraic vs. analytic De Rham cohomology}  
 Let now $X$ be the complement of a finite set $Z$ of points in a projective smooth curve $\bar X$ over $\C_p$. 
 Let ${\sM}$ be an algebraic coherent module on $X$ with connection $\nabla$. We denote by $({\sM}^{an}, \nabla^{an})$ the associated analytic coherent module with connection of the analytic curve $X^{an} $ attached to $X$.
 
 In this situation, $\sL^{an} ={\rm Ker}\, \nabla^{an}$ is no longer locally constant, not even constructible (for instance, in the case of the homogeneous differential equation of order $2$ satisfied by $\log z$ on $\Delta^\ast$, the dimension of the fibres of $\sL^{an}$ at points of type 2 drops to $1$). 
 
 According to F. Baldassarri, in the absence of Liouville numbers among Turrittin exponents (for instance, if $X$ and $\nabla$ are defined over a number field), formal solutions at $Z$ converge (the formal decomposition is analytic) \cite{Ba1} and the canonical morphism
 $$H^\ast_{DR}( \nabla)  \stackrel{\phi}{\to} 
 H^\ast_{DR}( \nabla^{an}) $$  
is an isomorphism \cite{Ba2} (this actually holds in any dimension, \cf \cite[ch. 4]{AB}).

 \begin{rem}  Since $\bar X^{an}$ can be covered by finitely many affinoids, it can be described as the completion of $\bar X(\C_p)$ with respect to a uniformity defined in terms of rational domains \cite[4.3, rem. 3]{Bak}, This leads to an interpretation of analytic De Rham cohomology as De Rham cohomology of a suitable uniform sheaf with connection on $X(\C_p)_{(an)}$.
   For coherent modules, working on analytic spaces or adic spaces (or in the rigid context) makes no difference, and one could also use the De Rham cohomology of $\nabla^{ad}$ (and a quasi-uniform interpretation on $X(\C_p)_{(ad)}$).
\end{rem}

   \subsubsection{The overconvergent setting} Let us assume that $\bar X$ comes from a projective smooth $\sO_{\C_p}$-scheme $\bar\sX$, with reduction $\bar\sX_{\bar\F_p}$, and that $Z$ comes from a \'etale closed subscheme $\sZ\subset \sX$. 
   
 It is then natural to consider the cohomology of $\nabla$ over the complement $V $ of the tube of $\sZ_{\bar\F_p}$, which is an affinoid domain in $X^{an}$ - or rather (to ensure finite-dimensionality), the cohomology of $\nabla$ on the corresponding ``dagger space". 
 As in \cite[5.1]{Berk2}, we see the latter as a germ of neighborhoods $V'$ of $V$ (say, in $X^{an}$), and associate to it the dagger algebra $$A^\dagger = {\rm{colim}}_{V'}\, \sO(V')$$  
(for instance, in the case of $\Delta = {\rm Spm} \,\C_p\langle T\rangle\subset \mathbb A^1$, $\C_p\langle T\rangle^\dagger = \bigcup_{r>1} \C_p\langle r^{-1}T\rangle$). One has ${\rm{Spm}}\, A^\dagger = V$.
  The category of sheaves on such germs is defined as the 2-colimit of the categories of sheaves of
sets on open neighborhoods of $V$ (\cf \cite[5.2]{Berk2} [SGA4, Exp. VI]; the situation is similar to that of \ref{ano}).
 
 \begin{rem} A simpler approach consists in using the overconvergent adic space $V^{ad\dagger}$ (obtained, as a set, by attaching to $V^{ad}$ one tangential base point for each point in $ Z $, \cf \ref{BH}). It is endowed with a sheaf of rings which coincides with $\sO_{V^{ad}}$ on ${V^{ad}}$, and has $A^\dagger$ as ring of global sections.
  One has $V^{ad\dagger }= \lim\, V'^{ad}$, whence a map $V^{ad\dagger}\to  ``\lim"\, V'$.  This allows to replace ind-sheaves on $``\lim"\, V'$ by genuine sheaves on $V^{ad \dagger}$.

 The \'etale site of $V^{ad\dagger}$ endowed with $\sO_{V^{ad\dagger}}$ is reminiscent of the subanalytic site endowed with the sheaf of tempered analytic functions (\ie analytic functions with moderate growth at the boundary) in the complex case, in its capacity of capturing  the exponentials $e^{Q(1/{z)}}$ which occur in the formal decomposition of differential modules \cite{Mo}. Note that the defining sets of the $e^{Q(1/{z)}}$'s ($p$-adic analogs of the petals of Deligne's daisies) are \'etale neighborhoods of $V^{ad\dagger}$ (\eg for the Dwork exponential $e^{\pi/z}$, it is the Artin-Scheier cover defined by $ y^{-p} - y^{-1} = z^{-1}$). 
   \end{rem}

  \subsubsection{Overconvergent De Rham cohomology}   
  
 Let $( M^{\dagger}, \nabla^{\dagger})$ be the $A^\dagger$-module with connection obtained from $(\sM, \nabla)$ by taking global sections and tensoring with $A^\dagger$.  There is a natural map
 
  $$H^\ast_{DR}( \nabla)  \stackrel{\phi}{\to} H^\ast_{DR}( \nabla^{\dagger}) = { H}^\ast (  {M}^\dagger \stackrel{\nabla^\dagger}{\to} {M}^\dagger \otimes_{\sO(X)} \Omega^1(X )).$$    
Let us assume that $( M^{\dagger}, \nabla^{\dagger})$ comes from an overconvergent $F$-isocrystal $\sE$ on $\sX_{\bar\F_p}$ (so that $H^\ast_{DR}( \nabla^{\dagger})= H^\ast_{rig}(\sE)$).  If $\nabla $ is regular, $\phi$ is an isomorphism (this actually holds in any dimension \cite{BC}), but it is not in general.

\smallskip The $p$-adic index formula \cite[5.0-12]{CM} concerns $H^\ast_{DR}(\nabla^\dagger)$ rather than $H^\ast_{DR}(\nabla)$ and involves the $p$-adic irregularities:
\begin{equation}\label{pchi}\chi_{DR}(\nabla^\dagger)= {\rm{rk}}\, \sM \cdot \chi_{rig}(\sX_{\bar\F_p})+ \sum_{x\in \sZ_{\bar\F_p}}\,{\rm{ir}}_x(\nabla^\dagger).\end{equation} 

  \begin{rem} The proof given in \loccit and summarized in \cite[4.4.1]{Ke}  is global, and uses at some point GAGA and Deligne's index formula over $\C$. It is used in K. Kedlaya's ``$p$-adic proof" of the Weil conjectures \cite[6.5.3]{Ke}... which is thus not completely $p$-adic! Fortunately, F. Baldassarri \cite{Ba4} has recently given a purely $p$-adic local proof of \eqref{pchi} in the spirit of Robba's original approach to the $p$-adic index formula. His result does not assume the overconvergence of $\nabla^\dagger$ (which is too restrictive a condition if one wishes to compare $\chi_{DR}( \nabla) $ and $\chi_{DR}( \nabla^{\dagger})$ by interpolation, using neighborhoods of $V$).  \end{rem}
  
 \smallskip The $p$-adic local monodromy theorem combined with the Madsuda-Tsuzuki-Crew identification of $p$-adic irregularities with Swan conductors (\cf \cite[7.1.2]{A2}) allows to write \eqref{pchi} in the form: 
  \begin{equation}\label{pchisw}\chi_{DR}(\nabla^\dagger)= {\rm{rk}}\, \sM \cdot \chi_{rig}(\sX_{\bar\F_p})+ \sum_{x\in \sZ_{\bar\F_p}}\,{\rm{sw}}_x(\sE)\end{equation} 
(\cf  \cite[4.4.1]{Ke} for details). For instance, if $\sE$ is a unit-root $F$-isocrystal, it corresponds to $p$-adic \'etale sheaf $\sL_{\bar\F_p}$ on $ \sX_{\bar\F_p}$ with finite local monodromy \cite{Cr}, one has $H_{DR}^\ast(\nabla^\dagger)\cong H^\ast(\sX_{\bar\F_p, et}, \sL_{\bar\F_p})$, and \eqref{pchisw} corrresponds to the Grothendieck-Ogg-Shafarevich formula for $\sL_{\bar\F_p}$. 

One might hope for an interpretation of \eqref{pchisw} more in the spirit of the combinatorial interpretation \ref{comb} of Deligne's index formula. In the special case where $\sL_{\bar\F_p}$ has finite global monodromy (which is for instance the case if ${\rm{rk}}\, \sM=1$ \cite{Cr}), $\nabla^\dagger$ is isotrival, and the Coleman-Berkovich \'etale sheaf $\sS$ \cite{Berk2}, or rather its avatar on $V^{ad\dagger}$, would serve as a substitute for the sheaf $\sO_{X_{(sec)}}^{mod}$: it restores the Poincar\'e lemma in the $p$-adic situation \cite[9.3]{Berk2} (\ie tensoring with it makes the connection an epimorphism), and the kernel of  ${\nabla^\dagger}_\sS$ is a locally constant \'etale sheaf whose cohomology coincides with $H^\ast_{DR}( \nabla^{\dagger})$.

\begin{rem} R. Huber \cite{Hu2} has attached to tangential base-points in $V^{ad\dagger}$ (corresponding to $\sZ_{\bar\F_p}$) some local monodromy groups, and used them successfully in the context of $\ell$-adic \'etale sheaves on $V^{ad\dagger}$, $\ell\neq p$ (\cf also \cite{Ra}). It would be very desirable to build a similar theory in the case of $p$-adic coefficients, especially in connection with the $p$-adic local monodromy theorem and the $p$-adic index formula.   \end{rem}

\bigskip
  \section*{Appendix 1. Uniformity and stratifications}
 
 \subsection{Uniformity, ``Grothendieck style"} 
   Let $\sC$ be a category. Let ${T} $ be an object of
$\sC$ whose finite powers exist in $\sC$.
An {\it entourage} of ${T}$ is given by an object $E$
and a triple of morphisms
${T}\stackrel{\delta}{\to} E, \; E\stackrel{p_1,p_2}{\to} {T}$ such that $p_i\circ \delta = 1_{T}$. 
 They form a category whose  final object is ${T}^2$.
      
 \begin{defn} A {\it uniformity} on ${T}$ is a sieve $\sU$ of entourages
of ${T}$, stable by  fibred product (over ${T}^2$), and
such that
 
- for any $E\in \sU$, there is $E'\in \sU$ such that
$(E'\times {T})\times_{{T}^3}({T}\times E')$ (together with $\delta$  and $p_1,p_3$)
exists in $\sU$ and maps to $E$,
 
- for any $E\in \sU$, the entourage obtained by
switching $p_1$ and $p_2$ is in $\sU$.  \end{defn}

  \begin{ex}
Let $S$ be a topological space, and $\sC_S$ be the category
of topological spaces ${T}$ with a surjective continuous
map to $S$. A uniformity then corresponds
to the usual notion, except that ${T}^2$ is
replaced by ${T}\times_S {T}$; the induced uniformity in
the  fibers should be compatible with the topology,
and vary continuously on $S$.

This example is fundamental in  {\it sectional representation
theory} (Dauns, Hofmann, et al. \cite{DH}),
which aims at a general recipe for representing
topological algebras $B$ as algebras of continuous
(or bounded continuous) sections of
objects of $\sC_S$ (for suitable ``spectra" $S$), endowed
with a uniformity.
This recipe generalizes both Grothendieck's
construction in the commutative discrete case
($B = \Gamma \sO_S$ viewed as sections of an \'etal\'e space)
and Gelfand's construction in the commutative
 $C^\ast$ case ($B = \Gamma ((\R\times S)/S)$).\end{ex}

  \subsection{Stratifications}
  Let $\sU$ be a uniformity on ${T}$, which we assume to be a ringed space (or a
topos).
  
 \begin{defn}  A {\it $\sU$-stratification} is an $\sO_{T}$-module $\sM$ together with
an isomorphism $p_1^\ast\sM\cong  p_2^\ast\sM$  (parallel transport)
over some $E\in \sU$, which $\delta^\ast$   maps to
$1_\sM$, and with the usual cocyle relation (on
$(E'\times {T})\times_{{T}^3}({T}\times E')$ for $E'$ as above). \end{defn}
   
  \begin{ex} Let ${T}$ be the complement of  finitely many
points $x_j$  in a projective complex curve   $\bar {T}$, 
and let $\sM$ be a coherent analytic module on ${T}$. 
   A connection $\nabla$ on $\sM$  induces a stratification on any simply-connected open $U\subset {T}$ 
  (e.g. a sector of angle
$<2\pi$  pointing at $x_j$).  If one replaces ${T}$ by its \'etale site, one gets a stratification without having to restrict to $U$.
    \end{ex}
  
  \begin{ex}   Grothendieck's $n$-stratifications
are $\sU$-stratifications for the uniform structure
defined by the $n$-th infinitesimal neighborhood
of the diagonal.
  Convergent and overconvergent isocrystals (\cf \cite{Bert}\cite[3.4]{LS}) are
$\sU$-stratifications for the (non-separated) uniform
structure defined by the tube of the diagonal (\cf \ref{form})
- or some strict neighborhoods thereof.  \end{ex}

\bigskip
  \section*{Appendix 2. Uniformity and bornology}

 In the mid 30s, bounded sets became important in functional analysis in the hands of von Neumann and Kolmogorov, who used them to define locally convex polar topologies. 
 Bornological rings and modules also play a crucial role in the above-mentioned sectional representation theory \cite{DH}. They have been recently reconsidered by F. Baldassarri as a possible common framework for complex-analytic geometry and overconvergent $p$-adic geometry.

 \begin{defn} A {\it bornology} on a set $X$ is a covering $(B_i)_i$ which is stable under inclusions and finite unions. The $B_i$'s are called {\it bounded sets}, and $(X, (B_i)_i)$ a {\it bornological space}.  A {\it bounded map} between bornological spaces is a map which sends bounded subsets to bounded subsets. 
  \end{defn}
 
 For instance, given a topological space $X$, the subsets with quasi-compact closure form a bornology, which is functorial in $X$.

\smallskip Any uniformity gives rise to two {\it bornologies}: the {\it precompact bornology} (consisting of the precompact subsets), and the (coarser) {\it canonical bornology} (consisting of subsets $B$ such that for any entourage $E$ there is a finite set $Z\subset X$ and a positive integer $n$ such that $E^{\circ n}(Z)$ contains $B$, \cite[II, \S 4, ex. 7]{Bou}). 
   Both play a crucial role in the classical context of topological spaces\footnote{in a locally convex space, the canonical bornology coincides with the von Neumann bornology (consisting of subsets absorbed by any neighborhood of the origin). An interesting critical evaluation of the role of bornologies in functional analysis may be found in \cite{Wa}.} -  \cf \cite{Me}, where (uniform) topological concepts are carefully compared to bornological concepts. 
 
More generally, any quasi-uniformity gives rise to a precompact bornology and a canonical bornology. One can then use functorial quasi-uniformities (such as the finest compatible quasi-uniformity) to attach functorially a bornology to a topology. It might be interesting to study this family of functors in the spirit of \cite{Br}.

  \end{sloppypar}

 \end{document}